\documentclass{article}%
\usepackage{amsmath, amsthm}
\usepackage{amsfonts}
\usepackage{amssymb}
\usepackage{enumerate}
\usepackage{graphicx}
\usepackage[table]{xcolor}
\usepackage[all,arc,curve,color,frame]{xy} 

\usepackage{amsmath}
\usepackage{amsfonts}
\usepackage{amssymb}
\usepackage{graphicx}%
\usepackage{color}
\usepackage{amsmath}
\usepackage{amsfonts}
\usepackage{amssymb}
\usepackage{graphicx}%
\usepackage{xypic}
\usepackage{tikz-cd}
\usepackage{color}
\usepackage{mathbbol}
\providecommand{\U}[1]{\protect\rule{.1in}{.1in}}
\newtheorem{theorem}{Theorem}

\newtheorem{corollary}[theorem]{Corollary}

\newtheorem{definition}[theorem]{Definition}
\newtheorem{example}[theorem]{Example}

\newtheorem{lemma}[theorem]{Lemma}

\newtheorem{proposition}[theorem]{Proposition}
\newtheorem{remark}[theorem]{Remark}

\tikzset{
  symbol/.style={
    draw=none,
    every to/.append style={
      edge node={node [sloped, allow upside down, auto=false]{$#1$}}}
  }
}

\title{The MF property for amalgamated free products}

\author{Tatiana Shulman}

\begin{document}

\maketitle

\begin{abstract}
A C*-algebra (or a group) is called MF (matricial field) if it admits finite-dimensional approximate unitary representations which are approximately injective, where approximately is understood with respect to the operator norm.

It is proved that for any MF C*-algebra $A$ and its C*-subalgebra $C$, $A\ast_C A$ is MF. For general amalgamated free products, $A\ast_C B$, a necessary and sufficient condition for being MF is given. It is shown that the following groups -- amalgamated free products of amenable groups, semidirect products of amenable groups by free groups, and $\mathbb Z^2\rtimes SL_2(\mathbb Z)$ -- all have MF full group C*-algebra.  It is  shown that the class of MF C*-algebras is closed under maximal tensor products with $C^*(\mathbb F_n)$.
\end{abstract}

\section{Introduction}
Approximation of arbitrary objects by well-understood ones is a recurring theme in almost any area of mathematics.  Approximation properties of C*-algebras lie at the heart of operator algebra theory.
MF C*-algebras are the ones that can be approximated, in an appropriate sense, by finite-dimensional ones. The MF property is closely related to several other important properties: it is implied by quasidiagonality and, in turn, implies stable finiteness.  It has an interesting connection with the Ext semigroup of Brown, Douglas and Fillmore (see e.g. \cite{HTh}).

The class of MF C*-algebras contains numerous important examples, and several constructions are known to preserve the MF property.
In particular Hadwin, Li and Shen proved that free products of MF C*-algebras are MF \cite{HadwinLiShen}. Later Li and Shen  discovered that if $A$ is a  separable  MF C*-algebra and $F$ is its  finite-dimensional C*-subalgebra, then $A\ast_F A$ is MF \cite{LiShen}.

  We prove here that, surprisingly, the finite-dimensionality assumption can be dropped entirely: the amalgamating C*-subalgebra can be arbitrary.  

\medskip

{\bf Theorem A} (Th. 10) {\it  Let $A$ be a separable  MF C*-algebra and let $C$ be any C*-subalgebra of $A$. Then $A\ast_C A$ is MF.  }

\medskip

For amalgamated free products $A\ast_F B$, where $A$ and $B$ are separable C*-algebras and $F$ is their common  finite-dimensional C*-subalgebra, Li and Shen found a necessary and sufficient condition  to be MF \cite{LiShen}: there are embeddings of $A$ and $B$ into $\prod M_n/\oplus M_n$ that are compatible on $F$. While this condition is obviously necessary for the amalgamated free product to be $MF$,  the sufficiency direction is highly nontrivial and constitutes the main part of Li and Shen's proof. 

 Our second main result  shows that the Li--Shen characterization remains valid for arbitrary amalgamated free products, with no finite-dimensionality assumption on the amalgamating subalgebra. Thus, we obtain   a complete characterization of when an amalgamated free product of separable C*-algebras is  MF.

\medskip

{\bf Theorem B} (Th. 20) {\it  Let $A, B, C$ be separable C*-algebras and $\theta_A: C \to A$, $\theta_B: C \to B$ be inclusions. Then $A\ast_C B$ is MF if and only if there exist embeddings $\phi_A: A \to \prod M_n /\oplus M_n$ and $\phi_B: B\to  \prod M_n /\oplus M_n$ such that $\phi_A\circ \theta_A = \phi_B\circ \theta_B$.  }

\medskip

We apply our results to investigate  MF property of full group C*-algebras and of groups themselves.

MF groups were introduced in \cite{CDE} as a group-theoretic analogue of  MF  C*-algebras. A question attributed to Kirchberg askes whether every discrete group is MF.\footnote{It is hard to trace the origin of this question. In \cite{BlackadarKirchberg} Blackadar and Kirchberg asked whether $C^*_r(G)$ is MF, for every discrete group $G$ but it is not the same as $G$ being MF, at least apriori.} This question is in the same spirit as other approximation problems for groups, such as the sofic conjecture and the Connes Embedding conjecture for groups. All three problems -- the sofic conjecture, the Connes Embedding conjecture for groups, and Kirchberg's MF-question -- state that every discrete group can be approximated by "nice" ones, namely by the permutation groups $Sym(n)$ and by the unitary groups $\mathcal U(n)$ of $n\times n$ matrices.



Relatively little is known about MF groups. Amenable groups -- which had long been known to be sofic and hyperlinear -- were shown to be MF only relatively recently in the breakthrough paper \cite{TWW}.

Since a discrete group $G$ embeds into $C^*(G)$ and $C^*_r(G)$, there is the following connection between the C*-algebraic MF property and the group theoretic one: if either $C^*(G)$ or $C^*_r(G)$ is MF, so is $G$.


The main source of groups with MF full C*-algebras is the class of amenable groups and  groups with RFD full group C*-algebras (see e.g. \cite{LS}, \cite{ShulmanSkalski} and \cite{RFDamalg} for examples of such). In the reduced case, important non-amenable examples include $C_r^*(\mathbb F_2)$, by the theorem of Haagerup and Thorbjornsen \cite{HTh}, for more recent  examples see e.g. \cite{SchafhauserAmalgamated}.



\medskip

An immediate consequence of Theorem A is  that if $C^*(G)$ is MF, then for any double $G\ast_H G$ of $G$, $C^*(G\ast_H G)$ is MF.
The next results are obtained using Theorem B.

\medskip

{\bf Theorem} (Th. 23) {\it If $G_1$ and $G_2$ are amenable groups and $H$ is their common subgroup, then $C^*(G_1\ast_H G_2)$ is MF.}

\medskip

In \cite{RSch} Rainone and Schafhauser proved that if $ G$ is amenable, then for any action of $\mathbb F_n$ on $G$, $C^*_r(G\rtimes \mathbb F_n)$ is MF.
Here we use our results on amalgamated free products to obtain a similar result for full group C*-algebras.

\medskip

{\bf Theorem} (Cor. 35) {\it For any semidirect product $G\rtimes  \mathbb F_n$ of an amenable group  $G$ by a free group $ \mathbb F_n$,  $C^*(G\rtimes \mathbb F_n)$ is MF.}

\medskip

\noindent We also prove that $C^*(\mathbb Z^2\rtimes SL_2(\mathbb Z))$ is MF  ({\bf Cor. 34}).


The same technique that we used for amalgamated free products applies to prove  that the class of MF C*-algebras is closed   under maximal tensor products with $C^*(\mathbb F_n)$ ({\bf Th. 28}) and under central HNN-extensions ({\bf Th. 29}).

We also obtain a result that provides new examples of groups whose full C*-algebras are residually finite-dimensional (RFD). Recall that RFD property  is stronger than MF. We prove that if  $G_1, G_2$ are amenable residually finite groups  and $H$ is a finite subgroup of both, then $C^*(G_1\ast_H G_2)$ is RFD ({\bf Cor. 23}).

\medskip

The key tool used for proving our main  results is the lifting characterization of the MF property established by the author in \cite{homotopyLifting}. This characterization allows us to reduce the MF problem for amalgamated free products to an appropriate lifting problem.

\bigskip

\noindent {\bf Acknowledgments. } The author is grateful to Don Hadwin for a useful discussion.
  The author is partially supported by a grant from the Swedish Research Council.

\section{Preliminaries}

\subsection{Amalgamated free products}

\begin{definition} Let $A, B, C$ be (unital) C*-algebras with (unital) embeddings $\theta_A: C \to A$ and $\theta_B: C \to B$. The (unital) amalgamated free product of $A$ and $B$ over $C$ is the C*-algebra $D$, equipped with (unital) embeddings $i_A: A \to D$ and $i_B: B\to D$ satisfying $i_A\circ \theta_A = i_B\circ \theta_B$, such that $D$ is generated by $i_A(A)\bigcup i_B(B)$ and satisfies the following universal property:

whenever $\mathcal E$ is a (unital) C*-algebra and $\pi_A: A \to \mathcal E$ and $\pi_B: B \to \mathcal E$ are (unital) $\ast$-homomorphisms satisfying $\pi_A\circ\theta_A = \pi_B\circ \theta_B$, there is  a (unital)  $\ast$-homomorphism $\pi: D \to \mathcal E$ such that $\pi\circ i_A = \pi_A$ and $\pi\circ i_B = \pi_B$.
\end{definition}

When $A=B$ we will use notation $i_1, i_2$ instead of $i_A, i_B$. 

The (unital) amalgamated free product $C$ is usually denoted by $A\ast_C B$.
Sometimes, when it is important to emphasize how $C$ is embedded into $A$ and $B$ respectively, we will write $A\ast_{\theta_A, \theta_B} B$   rather than $A\ast_C B$. When we write $A\ast_C A$ we always mean that $C$ is embedded into both copies of $A$ identically.

The amalgamated free product of more than two factors is defined analogously.

C*-algebraic definition of amalgamated free product agrees with the group-theoretic definition of amalgamated free product in the following sense: $$C^*(G_1\ast_{H} G_2) = C^*(G_1)\ast_{C^*(H)} C^*(G_2).$$


\subsection{MF C*-algebras and groups}
We use notation $M_n$ for the C*-algebra of all $n\times n$ matrices.
Let $$\prod M_n = \{(T_n)_{n\in \mathbb N}\;|\; T_n \in M_n, \;\sup_n \|T_n\|<\infty\},$$
$$\bigoplus M_n = \{(T_n)_{n\in \mathbb N}\;|\; \lim_{n\to \infty} \|T_n\| = 0\}.$$


\begin{definition} (\cite{BlackadarKirchberg}) A C*-algebra is {\it matricial field (MF)} if it embeds into  $\prod M_n/\oplus M_n$.
\end{definition}



 Equivalently, $A$ is MF if there exist maps $\phi_n: A \to M_{k_n}$, for some $k_n\in \mathbb N$, which are approximately multiplicative, approximately linear, approximately self-adjoint, and approximately injective. Reformulating this "locally", $A$ is MF if for any finite $F\subset A$ and $\epsilon >0$ there is $k$ and a map $\phi: A \to M_k$ such that
\begin{multline*}\|\phi_(a)\|>\|a\|-\epsilon, \;\|\phi(ab)- \phi(a)\phi(b)\|\le \epsilon,\\
\|\phi(a+b)- \phi(a)-\phi(b)\|\le \epsilon, \;\|\phi(a^*)- \phi(a)^*\|\le \epsilon,\end{multline*} for any $a, b \in F$.

\begin{definition} (\cite{CDE}) A group is {\it MF} if it embeds into the unitary group $\mathcal U \left(\prod M_n/\oplus M_n\right)$ of $\prod M_n/\oplus M_n$.\footnote{A different, more strong, property of a group was also called MF in \cite{SchafhauserAmalgamated}, see also \cite{de la Salle}.}
\end{definition}

A discrete group $G$ embeds into $C^*(G)$ and $C^*_r(G)$ and therefore, if either $C^*(G)$ or $C^*_r(G)$ is MF, so is $G$.

\subsection{Asymptotic homomorphisms}

{\bf Definition} (\cite{ConnesHigson}). An {\it asymptotic homomorphism} from $A$ to $B$ is a family of maps $(f_{\lambda})_{\lambda\in [0,\infty)}: A \to B$ satisfying the following properties:

\medskip

(i) for any $a\in A$, the mapping $[0, \infty) \to B$ defined by the rule $\lambda \to f_{\lambda}(a)$  is continuous;

(ii) for any $a, b\in A$ and $\mu_1, \mu_2\in \mathbb C$, we have

\begin{itemize}
\item  $\lim_{\lambda\to \infty} \|f_{\lambda}(a^*) - f_{\lambda}(a)^*\| = 0;$
\item  $\lim_{\lambda\to \infty} \|f_{\lambda}(\mu_1a + \mu_2b) - \mu_1f_{\lambda}(a) - \mu_2f_{\lambda}(b)\| = 0;$
\item $\lim{\lambda\to \infty}  \|f_{\lambda}(ab) - f_{\lambda}(a)f_{\lambda}(b)\| = 0.$
\end{itemize}

We will call $(f_{\lambda})_{\lambda\in \Lambda}: A \to B$, where $\Lambda$ is a directed set,   a {\it discrete asymptotic homomorphism} if the condition (ii) above is satisfied and for each $a\in A$ one has $\sup_{\lambda} \|f_{\lambda}(a)\|< \infty$. (For usual asymptotic homomorphisms the last condition holds automatically, see \cite{ConnesHigson} or \cite{Dadarlat2}). In this paper discrete asymptotic homomorphisms will be indexed by $\Lambda = \mathbb N$.

\medskip

 Let $q: B \to B/I$ be a quotient map and $f: A \to B/I$ a $\ast$-homomorphism. Following Manuilov and Thomsen (\cite{MT15}, \cite{Manuilov}), we say that $f$ {\it lifts to a discrete asymptotic homomorphism} $\phi_n$, $n\in \mathbb N$, if $q\circ \phi_n = f$, for each $n\in \mathbb N$.
 
 We say that $f$ {\it lifts asymptotically to a discrete asymptotic homomorphism} $\phi_n$, $n\in \mathbb N$, if $\lim_{n\to \infty} \| q\circ \phi_n(a)- f(a)\|=0$, for each $a\in A$.

\section{Amalgamated free products $A\ast_C A$}
Let $H$ be a Hilbert space and let $P_{n}$, $n \in \mathbb N$,  be an increasing sequence of projections of dimension $n$ that $\ast$-strongly converges to $1_{B(H)}$. We will identify $P_{n} B(H) P_{n}$ with $M_{n}$.
Let $\mathcal D \subset \prod_{n\in \mathbb N} M_{n}$ be the C*-algebra of all $\ast$-strongly convergent sequences of matrices. Let $q: \mathcal D \to B(H)$ be the surjection that sends each sequence  to its $\ast$-strong limit.

\medskip

Our main tool for exploring the MF-property of amalgamated free products will be the following lifting characterization of the MF-property obtained in \cite{homotopyLifting}.

\begin{theorem}\label{characterization} (Shulman \cite[Th. 14]{homotopyLifting})  Let $A$ be a separable C*-algebra. TFAE:

\begin{itemize}

\item[(i)] $A$ is MF,

\item[(ii)] every $\ast$-homomorphism from $A$ to $B(H)$ lifts to a discrete asymptotic homomorphism from $A$ to $\mathcal D$,


\item[(iii)] there exists an embedding of $A$ into $B(H)$ that  lifts asymptotically to a discrete asymptotic homomorphism from $A$ to $\mathcal D$.
\end{itemize}
\end{theorem}

\begin{remark} In the above characterization of the MF property  $\mathcal D$ and $ B(H)$ can be replaced by $M_N(\mathcal D)$ and $ B(H^{\oplus N})$ respectively,  for any $N\in \mathbb N$.
\end{remark}

\begin{remark}\label{contractive} In Theorem \ref{characterization} one can replace liftability to a discrete asymptotic homomorphism by liftability to a contractive discrete asymptotic homomorphism. To arrange this, in the proof of Theorem \ref{characterization} one uses the existence of a contractive continuous section proved in \cite{sectionsCones}.
 \end{remark}

\begin{lemma}\label{asHomAmalg} Let $A, B$ and $D$ be C*-algebras and let $C$ be a common C*-subalgebra of $A$ and $B$. Let $\phi_{\lambda}^A: A\to D$ and $\phi_{\lambda}^B: B \to D$, $\lambda\in \Lambda$, be contractive (discrete) asymptotic homomorphisms such that
$$\lim_{{\lambda}\to \infty}\|\phi_{\lambda}^A(c) - \phi_{\lambda}^B(c)\|=0,$$ for each $c\in C$. Then there exists a contractive (discrete) asymptotic homomorphism $\Phi_{\lambda}: A\ast_C B\to D$, ${\lambda}\in \Lambda$, such that
$$\lim_{{\lambda}\to \infty}\|\Phi_{\lambda}^A(a)- \phi_{\lambda}^A(a)\| =0,$$
$$\lim_{{\lambda}\to \infty}\|\Phi_{\lambda}^B(b)- \phi_{\lambda}^B(b)\| =0,$$
for each $a\in A, b\in B$.
\end{lemma}
\begin{proof} The asymptotic homomorphisms $\phi_{\lambda}^A: A\to D$ and $\phi_{\lambda}^B: B \to D$, $\lambda\in \Lambda$, induce $\ast$-homomorphisms $\phi^A: A \to C_b(\Lambda, D)/C_0(\Lambda, D)$ and $\phi^B: B \to C_b(\Lambda, D)/C_0(\Lambda, D)$ by formula
$$\phi_A(a) = \left(\phi_{\lambda}^A(a)\right)_{\lambda\in \Lambda} + C_0(\Lambda, D),$$
$$\phi_B(b) = \left(\phi_{\lambda}^B(b)\right)_{\lambda\in \Lambda} + C_0(\Lambda, D). $$ (Here $C_b(\Lambda, D)$ is the C*-algebra of all bounded continuous $D$-valued  functions on $\Lambda$ and $C_0(\Lambda, D)$ is its ideal consisting of all functions that vanish at infinity).  Since for each $c\in C$, $\lim_{{\lambda}\to \infty}\|\phi_{\lambda}^A(c) - \phi_{\lambda}^B(c)\|=0,$ we have
$$\phi^A(c) = \phi^B(c),$$
$c\in C$. Therefore $\phi^A$ and $\phi^B$ induce  a $\ast$-homomorphism $$\Phi = \phi^A\ast \phi^B: A\ast_C B \to C_b(\Lambda, D)/C_0(\Lambda, D).$$ By \cite[Th. 3]{sectionsCones} there is a contractive continuous section $s: C_b(\Lambda, D)/C_0(\Lambda, D) \to C_b(\Lambda, D)$. Then
$$\Phi_{\lambda} = ev_{\lambda}\circ s\circ \Phi: A\ast_C B \to D$$ is a contractive (discrete) asymptotic homomorphism. Since $\left(\Phi_{\lambda}(a)\right)_{\lambda\in \Lambda}$ and $\left(\phi_{\lambda}^A(a)\right)_{\lambda\in \Lambda}$ are both representatives of $\phi^A(a)$, we conclude that $$\lim_{{\lambda}\to \infty}\|\Phi_{\lambda}^A(a)- \phi_{\lambda}^A(a)\| =0,$$ for each $a\in A$. Similarly for $B$.
\end{proof}

\begin{lemma}\label{qau} Let $q: B \to B/I$ be  a quotient map  and let $\{i_{\lambda}\}_{\lambda\in \Lambda}$ be a quasicentral approximate unit for $I\lhd B$. Then for any $x, y\in B$  and $N\in \mathbb N$
$$\limsup_{\lambda} \|\lbrack (x(1-i_{\lambda}))^N, y\rbrack\| = \|\lbrack q(x)^N, q(y)\rbrack\|.$$
\end{lemma}
\begin{proof} Since $\{i_{\lambda}\}_{\lambda\in \Lambda}$ is quasicentral,
\begin{equation}\label{limsup}\limsup_{\lambda} \|\lbrack (x(1-i_{\lambda}))^N, y\rbrack\| = \limsup_{\lambda} \|(1-i_{\lambda})^N\lbrack x^N, y\rbrack \|.\end{equation}
Let $u_{\lambda} = 1 - (1-i_{\lambda})^N.$ Then $\{u_{\lambda}\}_{\lambda\in \Lambda}$ is itself a (quasicentral) approximate unit.
It is well-known that for any approximate unit, hence for $\{u_{\lambda}\}_{\lambda\in \Lambda}$, and any $z\in B$
$$ \limsup_{\lambda} \|(1-u_{\lambda})z\| = \|q(z)\|$$
(see e.g. \cite{LoringBook}). This and (\ref{limsup}) imply the statement.
\end{proof}

\begin{lemma}\label{crucial} Let $u, a_1, \ldots, a_N\in B/I,$ where $u$ is a unitary and $[a_i, u]=0, $ $i=1, \ldots, N$. Let $A_1, \ldots, A_N\in B$ be lifts of $a_1, \ldots, a_N$, $i=1, \ldots, N$, and let $\epsilon >0$. Then $\left(\begin{array}{cc} u &\\& -u^*\end{array}\right)$ lifts to a unitary $V\in M_2(B)$ such that $$\left\|\left[\left(\begin{array}{cc}A_i &\\& A_i\end{array}\right), V\right]\right\|\le \epsilon,$$ $i=1, \ldots, N$.
\end{lemma}
\begin{proof}
 First we lift $u$ to a contraction $X\in B$. Let  $\{i_{\lambda}\}, \lambda \in \Lambda,$ be a quasicentral approximate unit for $I\lhd B$.
 By Lemma \ref{qau} there exists $\lambda$ such that for $\tilde X = X(1-i_{\lambda})$ we have
$$\|[\tilde X, A_i]\|\le \frac{\epsilon}{2},$$
$$\|[\tilde X^*, A_i]\|\le \frac{\epsilon}{2},$$
$$\|[\tilde X \tilde X^*, A_i]\|\le \frac{\epsilon^2}{16},$$
$$\|[\tilde X^*\tilde X, A_i]\|\le \frac{\epsilon^2}{16},$$
$i=1, \ldots, N$. We define a unitary $V\in M_2(B)$ by
$$V = \left(\begin{array}{cc} \tilde X & \sqrt{1 - \tilde X\tilde X^*} \\ \sqrt{1 - \tilde X^*\tilde X} & -\tilde X^*\end{array}\right).$$
Then
$$\left[V, \left(\begin{array}{cc} A_i & 0\\0 & A_i\end{array}\right)\right] = \left(\begin{array}{cc}
[\tilde X, A_i] & \lbrack (1-\tilde X{\tilde X}^*)^{1/2}, A_i\rbrack \\  \lbrack (1-\tilde X^*\tilde X)^{1/2}, A_i\rbrack & [-\tilde X^*, A_i] \end{array}\right),$$ $i=1, \ldots, N$.

By Pedersen's inequality (which says that $\|[A^{1/2}, B]\|\le \frac{5}{4} \|[A, B]\|^{1/2}$ when $0 \le A \le 1$) we have
$$\|\lbrack (1-\tilde X{\tilde X}^*)^{1/2}, A_i\rbrack\| \le \frac{5}{4}\|\lbrack 1- \tilde X\tilde X^*, A_i\rbrack\|^{1/2}  = \frac{5}{4}\|\lbrack  \tilde X\tilde X^*, A_i\rbrack\|^{1/2} < \epsilon/2, $$ $i=1, \ldots, N$, and the same for $\lbrack (1-\tilde X{\tilde X}^*)^{1/2}, A_i\rbrack\|$.
Therefore
\begin{multline}\left\|\left[V, \left(\begin{array}{cc} A_i & 0\\ 0& A_i\end{array}\right)\right]\right\| \le \\ \left\|\left(\begin{array}{cc}
[\tilde X, A_i] &  0\\  0 & [-\tilde X^*, A_i] \end{array}\right)\right\| + \\ \left\|\left(\begin{array}{cc}
  0& \lbrack (1-\tilde X{\tilde X}^*)^{1/2}, A_i\rbrack \\  \lbrack (1-\tilde X^*\tilde X)^{1/2}, A_i\rbrack & 0 \end{array}\right)\right\|
  < \epsilon/2 + \epsilon/2 = \epsilon, \end{multline}
  $i=1, \ldots, N$.

  \end{proof}

\begin{theorem}\label{amalgamatedProduct} Let $A$ be a separable  MF C*-algebra and let $C$ be any C*-subalgebra of $A$. Then $A\ast_C A$ is MF.
\end{theorem}
\begin{proof} Let $\pi = \pi_1\ast\pi_2: A\ast_C A \to B(H)$ be an embedding. Then 
\begin{equation}\label{clarification} \pi_1|_C = \pi_2|C.\end{equation}
 Let
$$\rho_1 = \pi_1\oplus \pi_2, \; \rho_2 = \pi_2\oplus \pi_1.$$
Then $(\rho_1\oplus \rho_1)\ast (\rho_2\oplus \rho_2): A\ast_C A \to B\left(H^{\oplus 4}\right)$ is also an embedding. We will show that it asymptotically lifts to a discrete asymptotic homomorphism from $A\ast_C A$ to $M_4(\mathcal D)$. Theorem \ref{characterization}  then finishes the proof.

Since $\rho_1 = u^*\rho_2 u$, where $u = \left(\begin{array}{cc} 0 & 1\\1 & 0\end{array}\right)$, and since, as follows from (\ref{clarification}), $\rho_1|_C = \rho_2|_C$, we have
$$\rho_2(c) = u^*\rho_2(c)u,$$
$c\in C$, or, equivalently,
$$[\rho_2(c), u]=0,$$
$c\in C$.  Since $A$ is MF, by Theorem \ref{characterization} and Remark \ref{contractive} the representation $\rho_2$ lifts to a contractive discrete asymptotic homomorphism $\phi_n: A\to M_2(\mathcal D)$, $n\in \mathbb N$.  Let  $\{c_1, c_2, \ldots\}$ be a dense subset of $C$ and let  $F_k = \{c_1, \ldots, c_k\}$.
By Lemma \ref{crucial} there exists a lift $V_k\in M_4(\mathcal D)$ of $\left(\begin{array}{cc} u &\\& -u^*\end{array}\right)$ such that
\begin{equation}\label{commutator}\|\left[V_k, \left(\begin{array}{cc} \phi_k(c) & 0\\ 0& \phi_k(c)\end{array}\right)\right]\| \le \frac{1}{k}, \end{equation}
$c\in F_k$.

We define  asymptotic homomorphisms $\phi_k^{(1)}, \phi_k^{(2)}: A \to  M_4(\mathcal D)$, $k\in \mathbb N$, by
  $$\phi_k^{(2)} = \phi_k\oplus \phi_k, \; \phi_k^{(1)} = V_k^*\phi_k^{(2)}V_k.$$
  Then $$q\circ \phi_k^{(2)} = \rho_2\oplus \rho_2,$$
  $$ q\circ \phi_k^{(1)} = (u \oplus (-u^*))^* (\rho_2\oplus \rho_2)(u\oplus (-u^*) = (u \oplus u)^* (\rho_2\oplus \rho_2)(u\oplus u) =
  \rho_1\oplus \rho_1$$
  (we used here that $u^*=u$).

    Now we will show that the asymptotic homomorphisms $\phi_k^{(1)}, \phi_k^{(2)}$ asymptotically agree on $C$.
  Let $c\in C, \epsilon > 0$. There exists $k_0\in \mathbb N$ and $c_0\in F_{k_0}$ such that $\frac{1}{k_0} < \epsilon$ and $\|c-c_0\|< \epsilon.$
  Since $\phi_k^{(1)}, \phi_k^{(2)}$ are contractive asymptotic homomorphisms,
  there exists $k_1> k_0 $ such that for any $k> k_1$
  $$\|\phi_k^{(i)}(c) - \phi_k^{(i)}(c_0)\|\le \| \phi_k^{(i)}(c-c_0)\|+\epsilon \le \|c-c_0\| + \epsilon < 2\epsilon, $$
  $i = 1, 2$. Then, using (\ref{commutator}), for any $k> k_1$ we obtain
  \begin{multline*}\|\phi_k^{(1)}(c) - \phi_k^{(2)}(c)\| < 4 \epsilon + \|\phi_k^{(1)}(c_0) - \phi_k^{(2)}(c_0)\| \\
   = 4 \epsilon + \|V_k^*\phi_k^{(2)}(c_0)V_k - \phi_k^{(2)}(c_0)\| = 4 \epsilon + \|\lbrack V_k, \phi_k^{(2)}(c_0)\rbrack\|\le 4\epsilon + \frac{1}{k_0}< 5 \epsilon.\end{multline*}
   Thus, for each $c\in C$,
  $$\lim_{k\to \infty} \|\phi_k^{(1)}(c) - \phi_k^{(2)}(c)\|=0.$$ By Lemma \ref{asHomAmalg} there exists  an asymptotic homomorphism $\Phi_k: A\ast_C A\to M_4(\mathcal D)$, $k\in \mathbb N$,  such that
  $$\lim_{k\to \infty} \|\Phi_k(i_1(a))- \phi_k^{(1)}(a)\|=0, $$
  $$\lim_{k\to \infty} \|\Phi_k(i_2(a))- \phi_k^{(2)}(a)\|=0, $$
  $a\in A$. Applying  the quotient map $q$  to these equalities we obtain that
  $\Phi_k|_{i_1(A)}$, $k\in \mathbb N$,  is an asymptotic lift of $\rho_1\oplus \rho_1$ and $\Phi_k|_{i_2(A)}$, $k\in \mathbb N$, is an asymptotic lift of $\rho_2\oplus \rho_2$. Then,  for every $x$ in the dense subset  $$E = span \{i_1(a_1)i_2(a_2)\ldots i_1(a_n)i_2(a_{n+1})\;|\;n\in \mathbb N, a_i\in A\}$$ of $A\ast_C A$, $\|q\circ \Phi_k(x)- (\rho_1\oplus \rho_1)\ast (\rho_2\oplus\rho_2)(x)\|\to 0$ as $k\to \infty$.
   It implies that $\Phi_k$, $k\in \mathbb N$, is an asymptotic lift of $(\rho_1\oplus \rho_1)\ast (\rho_2\oplus\rho_2)$. By Theorem \ref{characterization}, $A\ast_C A$ is MF.
\end{proof}

Recall that a group is called {\it maximally almost periodic (MAP)} if it has a separating family of finite-dimensional representations. E.g. any residually finite group is MAP.

A C*-algebra is called  {\it residually finite-dimensional (RFD)} if it has a separating family of finite-dimensional representations.

\begin{example}\label{HigmanExample} In \cite{Higman} Higman proved that the group
$$G = \left<a, b, c\;| \; a^{-1}ca = c^2 = b^{-1}cb \right>$$ is not residually finite. Since it is finitely generated, it is also non-maximally almost periodic (non-MAP) by Maltsev's Theorem. This group can be written as the group double
$$G = BS(1, 2)\ast_{\mathbb Z} BS(1, 2),$$ where $BS(1, 2) = \left<a,  c\;| \; a^{-1}ca = c^2 \right>$ is the Baumslag-Solitar group and  $\mathbb Z = \left< c\right>$.
Then $C^*(G) = C^*(BS(1, 2))\ast_{C(\mathbb T)} C^*(BS(1, 2)).$ Since  $BS(1, 2)$ is amenable and therefore has MF full group C*-algebra (\cite{TWW}), by Theorem \ref{amalgamatedProduct} $C^*(G)$ is MF.
\end{example}

\begin{example} A group double of a residually finite (RF) group need not be RF even when amalgamation is done over a central subgroup (\cite{RFDamalg}). Namely for the Abels group
 $$\Gamma = \left\{ \left(\begin{array}{cccc}1&x_{12}&x_{13}&x_{14}\\ 0 &p^k&x_{23}&x_{24}\\0&0&p^n&x_{34}\\0&0&0&1 \end{array}\right) : x_{ij}\in \mathbb Z\left[\frac{1}{p}\right] , k, n\in \mathbb Z\right\}$$  and its central subgroup
$$N = \left\{ \left(\begin{array}{cccc}1&0&0&x \\0&1&0&0\\ 0&0&1&0\\ 0&0&0&1 \end{array}\right) : x\in \mathbb Z \right\} \cong \mathbb Z,$$ the amalgamated free product   $\Gamma\ast_{N} \Gamma$ is not maximally almost periodic, hence is not RF (\cite[Prop.8.]{RFDamalg}). Since $\Gamma$ is amenable RF and therefore $C^*(\Gamma)$ is RFD, again by Theorem \ref{amalgamatedProduct} we see that $C^*(\Gamma\ast_{N} \Gamma)$ is MF.
\end{example}

\subsection{General amalgamated free products}

Here we find necessary and sufficient conditions for general amalgamated free products to be MF.

\begin{theorem}\label{inclusion} (\cite[Th. 4.11]{LP}) Let $A, B, D$ be unital C*-algebras and $C$ be a separable unital C*-algebra. Let ${\theta}_A: C \to A$, ${\theta}_B: C \to B$,  and $\phi_A: A \to D$, $\phi_B: B \to D$ be unital inclusions such that $\phi_A\circ {\theta}_A= \phi_B\circ {\theta}_B$. Then $A\ast_C B$ embeds into $D\ast_C D$.
\end{theorem}

\begin{remark}\label{howEmbeddingLooks} In fact the embedding in the theorem above is given by the $\ast$-homomorphism $(i_1\circ\phi_A)\ast (i_2\circ \phi_B)$ that sends $A$ to the first copy of $D$ via $\phi_A$ and $B$ to the second copy of $D$ via $\phi_B$. (This $\ast$-homomorphism is well-defined since $\phi_A\circ \theta_A = \phi_B\circ \theta_B$).
\end{remark}

\begin{lemma}\label{unitization} (\cite[Lemma. 4.4]{LP}) Let $A^+$ denote the forced unitization of $A$. Then for any $A, B, C$
$$A^+\ast_{C^+} B^+ = (A\ast_C B)^+.$$
\end{lemma}

\begin{theorem}\label{general} Let $A, B, C$ be separable C*-algebras and $\theta_A: C \to A$, $\theta_B: C \to B$ be inclusions. Then $A\ast_C B$ is MF if and only if there exist embeddings $\phi_A: A \to \prod M_n /\oplus M_n$ and $\phi_B: B\to  \prod M_n /\oplus M_n$ such that $\phi_A\circ \theta_A = \phi_B\circ \theta_B$.
\end{theorem}
\begin{proof} "Only if" is clear. We prove the "if" part. First, we consider the unital case. Let
$$D = C^*(\phi_A(A), \phi_B(B)) \subset  \prod M_n /\oplus M_n.$$  By Theorem \ref{inclusion} $A\ast_C B \subset D\ast_C D$. Since $D$ is a separable MF C*-algebra, by Theorem \ref{amalgamatedProduct} $D\ast_C D$ is MF. Therefore $A\ast_C B$ is MF.

In the non-unital case (that is, when any of the C*-algebras or the inclusions is non-unital) we consider the forced unitizations of $A, B, C$ and the corresponding maps. By what was proved above, $A^+\ast_{C^+} B^+$ is MF. Lemma \ref{unitization} finishes the proof.
\end{proof}

Now we consider amalgamated free products of more than two factors. If $\theta_j: C \to A_j$, $j=1, \ldots, N$, are inclusions, then the corresponding amalgamated free product will be denoted by $\ast_{j=1}^N (A_j, \theta_j)$. Let $i_j: A_j\to \ast_{j=1}^k (A_j, \theta_j)$ denote the canonical inclusion.  Using the universal property of amalgamated free products, it is straightforward to check that
$$\ast_{j=1}^N (A_j, \theta_j) = \left(\ast_{j=1}^{N-1} (A_j, \theta_j)\right) \ast_{i_i\circ \theta_1, \theta_N } A_N.$$

\begin{lemma}\label{manySameFactors} If $A$ is a separable MF C*-algebra and $\theta: C\to A$ is an inclusion, then $\ast_{j=1}^N (A, \theta)$ is MF.
\end{lemma}
\begin{proof} By induction. Suppose the statement is proved for $k<N$. We write
$$\ast_{j=1}^N (A, \theta) = \left(\ast_{j=1}^{N-1} (A, \theta)\right) \ast_{i_i\circ \theta, \theta } A.$$ By the induction assumption  there is an embedding $\phi: \ast_{j=1}^{N-1} (A, \theta) \to \prod M_n/\oplus M_n$. We define $\phi_N: A \to \prod M_n/\oplus M_n$ by
$$\phi_N = \phi\circ i_1.$$ Then
$\phi_N\circ\theta = \phi\circ\i_1\circ\theta$. By Theorem \ref{general} $\ast_{j=1}^N (A, \theta)$ is MF.
\end{proof}

\begin{theorem}\label{manyFactors} Let $A_j,$ $j=1, \ldots, N$, and  $C$ be separable C*-algebras and $\theta_j: C \to A_j$, $j=1, \ldots, N$, be inclusions. Then $\ast_{j=1}^N (A_j, \theta_j)$ is MF if and only if there exist embeddings $\phi_j: A_j \to \prod M_n /\oplus M_n$, $j=1, \ldots, N$, such that $\phi_j\circ \theta_j = \phi_k\circ \theta_k$, for all $j, k\le N$.
\end{theorem}
\begin{proof} "Only if" is clear. We prove the "if" part by induction.
 Lemma \ref{unitization} holds for more than two factors and therefore, similar to the proof of Theorem \ref{general}, it is sufficient to consider the unital case. To avoid cumbersome notation, we will show how to prove the statement for three factors. For four or more factors it is similar.
 We have $$\ast_{j=1}^3 (A_j, \theta_j) = (A_1\ast_{\theta_1, \theta_2} A_2) \ast_{i_1\circ\theta_1, \theta_3} A_3.$$
 By Remark \ref{howEmbeddingLooks}, $(i_1\circ\phi_1)\ast(i_2\circ\phi_2): A_1\ast_{\theta_1, \theta_2} A_2 \to \left(\prod M_n/\oplus M_n\right)\ast_{\phi_1\circ\theta_1(C)} \left(\prod M_n/\oplus M_n\right)$ is an inclusion. We define $\tilde \phi_3: A_3\to  \left(\prod M_n/\oplus M_n\right)\ast_{\phi_1\circ\theta_1(C)} \left(\prod M_n/\oplus M_n\right)$ by $$\tilde\phi_3 = i_1\circ \phi_3.$$ Then
 $$\tilde \phi_3\circ\theta_3(c) = i_1(\phi_3(\theta_3(c))) = i_1(\phi_1(\theta_1(c))) = (i_1\circ\phi_1)\ast(i_2\circ\phi_2)(i_1\circ\theta_1(c)). $$
 By Lemma \ref{manySameFactors} there is an embedding $$\alpha: \left(\prod M_n/\oplus M_n\right)\ast_{\phi_1\circ\theta_1(C)} \left(\prod M_n/\oplus M_n\right) \to \prod M_n/\oplus M_n.$$ Then $\alpha\circ\tilde \phi_3$ and $\alpha\circ ((i_1\circ\phi_1)\ast(i_2\circ\phi_2))$ satisfy the necessary and sufficient condition of Theorem \ref{general}, and therefore $\ast_{j=1}^3 (A_j, \theta_j)$ is MF.
\end{proof}

\section{Amalgamated free products of amenable groups}

The next theorem shows that for amalgamated free products of amenable groups the necessary and sufficient condition of Theorem \ref{general} holds automatically.

\medskip

\begin{theorem}\label{amenable} If $G_1$ and $G_2$ are amenable groups, then $C^*(G_1\ast_H G_2)$ is MF.
\end{theorem}
\begin{proof} Let $\mathcal Q$ denote the UHF-algebra $\otimes_{d=1}^{\infty} M_d$. By \cite[Th. 2.5]{SchafhauserAmalgamated} there exist unital embeddings $\rho_i: C^*(G_i)\to \mathcal Q$, $i=1, 2$, such that
$$tr_{\mathcal Q} \rho_1(g_1) = 0 = tr_{\mathcal Q} \rho_2(g_2), $$
for any $g_1\in G_1\setminus \{1\}, g_2\in G_2\setminus \{1\}$. In particular
$$tr_{\mathcal Q} \rho_1(\theta_1(h)) = 0 = tr_{\mathcal Q}\rho_2(\theta_2(h)), $$for any $h\in H\setminus\{1\}$.  By \cite[Th. 2.5]{SchafhauserAmalgamated} there exist unitaries $u_n\in \mathcal Q$ such that
$$\rho_1(\theta_1(h)) = \lim_{n\to \infty} u_n^*\rho_2(\theta_2(h))u_n, $$ for any $h\in H$. Let $\mathcal Q_{\omega}$ be the norm ultraproduct of $\mathcal Q$ and let $q: \prod \mathcal Q\to \mathcal Q_{\omega}$ be the canonical surjection. We define embeddings
$ \tilde\rho_i: C^*(G_i)\to \mathcal Q_{\omega}$, $i=1, 2$,  by
$$\tilde\rho_1= q\circ \left(\rho_1\right)_{n\in \mathbb N},$$
$$\tilde\rho_2 = q\circ \left(u_n^*\rho_2u_n\right)_{n\in \mathbb N}.$$
Then $$\tilde\rho_1\circ\theta_1(h) = \tilde\rho_2\circ \theta_2(h), $$
$h\in H$. Since MF is a local property, $\mathcal Q_{\omega}$ is an MF-algebra \cite[Prop. 3.2 (3)]{RSch}. Let $\alpha: \mathcal Q_{\omega}\to \prod M_n/\oplus M_N$ be an embedding.  Let $\phi_i= \alpha\circ \tilde\rho_i, $ $i=1, 2$. Then $\phi_1, \phi_2$ satisfy the necessary and sufficient condition of Theorem \ref{general}. Therefore  $C^*(G_1\ast_H G_2)$ is MF.
\end{proof}

\begin{remark}\label{remarkManyAmenbaleFactors}  By the same arguments and Theorem \ref{manyFactors}, Theorem \ref{amenable} generalizes to amalgamated free products of arbitrarily many amenable groups.
\end{remark}

\begin{example} Let us consider the Baumslag-Solitar group $BS(1, n) = \left< a, t| t^{-1}at=a^n\right>.$
Out of two copies of $BS(1, n)$ one can form three natural examples of free products amalgamated over an infinite cyclic group:
$$G_1= BS(1, n)\ast_{\left<t_1=t_2\right>} BS(1, n),$$
$$G_2 = BS(1, n)\ast_{\left<a_1=a_2\right>} BS(1, n),$$
$$G_3 = BS(1, n)\ast_{\left<a_1=t_2\right>} BS(1, n).$$
All three C*-algebras $C^*(G_i)$, $i=1, 2, 3$, are MF by Theorem \ref{amenable} (for $C^*(G_1)$ and $C^*(G_2)$ this follows already from Theorem \ref{amalgamatedProduct}). In fact
$C^*(G_1)$ is  RFD  \cite[Prop. 5.9]{ShulmanSkalski}.

The group $G_3$ is a building block for the famous Higman's group $H \cong G_3\ast_{\mathbb F_2} G_3.$ Higman's group has always been considered as a candidate for counterexample to the approximation conjectures. Now at least we  can see that its building blocks are MF.
\end{example}

 \medskip

Below we find some sufficient condition for amalgamated free products of groups to have RFD full group C*-algebra.

The full group C*-algebra of a MAP group might not be RFD as discovered by Bekka \cite{Bekka99}. In \cite{RFDamalg} it was proved that for a large class of amenable groups (e.g. all polycyclic-by-finite groups) free products with amalgamation over central subgroups are RFD. Here we show that amalgamated free products of amenable RF groups with amalgamation over finite subgroups are RFD.

Below, given a homomorphism $\theta: H\to G$, the induced $\ast$-homomorphism $C^*(H) \to C^*(G)$ will also be denoted by $\theta$.

\begin{proposition}\label{embIntoProdMatrices} Suppose $G_1$ and $G_2$ are amenable, $\theta_k: H\to G_k$, $k=1, 2$, are embeddings, and the corresponding amalgamated free product $G_1\ast_H G_2$ is MAP. Then there exist embeddings $j_k: C^*(G_k)\hookrightarrow \prod M_n$ such that $j_1\circ \theta_1 = j_2\circ \theta_2$.
\end{proposition}
\begin{proof} Since $G_1\ast_H G_2$ is MAP, the delta function on it can be written as
$$\delta_e = \lim_{n\to \infty} tr\; \pi_n,$$ for some finite-dimensional representations $\pi_n$ of $G_1\ast_H G_2$ (\cite[Prop.3]{HadwinShulman}).  Then \begin{equation}\label{MAP1}\delta_e|_{i_k(G_k)} = \lim_{n\to \infty} tr\; \pi_n|_{i_k(G_k)},\end{equation}
$k =1, 2$. Since $G_k$ is amenable, the delta function on it extends to a faithful trace on $C^*(G_k)$. Therefore, (\ref{MAP1}) implies that
$$j_k:= \prod_n \pi_n|_{i_k(G_k)}: C^*(G_k) \to \prod M_n$$ is injective. We have
$$j_1\circ \theta_1(h) = \prod _n \pi_n(i_1(\theta_1(h))) = \prod _n \pi_n(i_2(\theta_2(h))) = j_2\circ \theta_2(h),$$
$h\in H$, and therefore $j_1\circ \theta_1(x) = j_2\circ \theta_2(x)$, for any $x\in C^*(H)$.
\end{proof}

\begin{corollary} If $G_1, G_2$ are amenable residually finite groups  and $H$ is a finite subgroup of them, then $C^*(G_1\ast_H G_2)$ is RFD.
 \end{corollary}
 \begin{proof}  Since $G_1, G_2$ are residually finite and $H$ is finite,  $G_1\ast_H G_2$ is RF (\cite[Th. 3]{Baumslag}), hence is MAP.  By Proposition \ref {embIntoProdMatrices} there exist embeddings $j_k: C^*(G_k)\hookrightarrow \prod M_n$ such that $j_1\circ \theta_1 = j_2\circ \theta_2$. Since $C^*(H)$ is finite-dimensional, by \cite{LiShenRFD} $C^*(G_1\ast_H G_2)$ is RFD.
 \end{proof}

\section{Central HNN-extensions and maximal tensor product with $C^*(F_n)$}

Below $\otimes$  stands for the maximal tensor product.

\begin{theorem}\label{tensorProduct} Let $N\in \mathbb N \cup \{\infty\}$. If a C*-algebra $A$ is MF, then $A\otimes C^*(F_N)$ is MF.
\end{theorem}
\begin{proof} Since MF is  a local property, we can assume that $A$ is separable. Let $\pi: A\otimes C^*(F_N) \to B(H)$ be an embedding. We define  a representation $\pi': A\otimes C^*(F_N)\to B(H)$ by
$$\pi'(a)=\pi(a), \; a\in A,$$
$$\pi'(u_i) = - \pi(u_i)^*, $$
for each generator $u_i$ of $C^*(F_N)$. Let $\tilde\pi = \pi\oplus \pi'$. Since $A$ is MF, by Theorem \ref{characterization} $\pi|_A$ lifts to a discrete asymptotic homomorphism $\phi_k: A \to \mathcal D$, $k\in \mathbb N$. Then $\tilde\pi|_A$ lifts to the discrete asymptotic homomorphism
$\tilde\phi_k = \phi_k\oplus \phi_k,$ $k\in \mathbb N$. Let $\{a_1, a_2, \ldots\}$ be a dense subset of $A$. By Lemma \ref{crucial} for each $k\in \mathbb N$ there is a lift $V_{i, k}$ of  $\tilde\pi(u_i)$, $i\le N$, such that
$$\|[V_{i, k}, \tilde\phi_k(a_l)]\|<\frac{1}{k}, $$
$l\le k$.  Thus $V_{i, k}$'s and $\tilde\phi_k(a)$'s, $a\in A$,  asymptotically satisfy the relations of $A\otimes C^*(F_N)$. By \cite[Lemma 7 + Rem. 8]{sectionsCones}, there exists a discrete asymptotic homomorphism from $A\otimes C^*(F_N)$ to $M_2(\mathcal D)$ that lifts $\tilde \pi$. By Theorem \ref{characterization}, $A\otimes C^*(F_N)$ is MF.
\end{proof}

\bigskip

Let $A$ be a unital $C^*$-algebra, $B$ and $C$ its C*-subalgebras, and $\phi: B\to C$ an isomorphism. The corresponding {\it HNN-extension}
$\langle A, t\;|\; t^{-1}Bt = C, \; \phi\rangle$ is a unital C*-algebra with the following properties:

1) There exists a unital $\ast$-homomorphism $i_A: A \to \langle A, t\;|\; t^{-1}Bt = C, \; \phi\rangle$ and a unitary $t\in \langle A, t\;|\; t^{-1}Bt = C, \; \phi\rangle$ such that
$$t^{-1}i_A(b)t= i_A(\phi(b)),$$ for any $b\in B$.

2) For any unital C*-algebra $D$ and any unitary $u\in D$ and unital $\ast$-homomorphism $\pi: A \to D$ such that $u^{-1}\pi(b)u = \pi(\phi(b))$, for any $b\in B$, there is a unique $\ast$-homomorphism $\sigma: \langle A, t\;|\; t^{-1}Bt = C, \; \phi\rangle \to D$ such that
$\sigma\circ i_A = \pi, \; \sigma(t) = u.$

\medskip

\noindent Such C*-algebra exists and is unique up to isomorphism.
An HNN-extension is {\it central} if    $B=C$ and $\phi=id$.  We will show that the MF-property passes to central HNN-extensions.

Below we identify $A$ with $i_A(A)$ to simplify notation.

\begin{theorem}\label{HNN} Let $A$ be an MF-algebra and $B$ its C*-subalgebra. Then the  HNN-extension $\langle A, t\;|\; t^{-1}Bt = B, \; id \rangle$ is MF.
\end{theorem}
\begin{proof} Similar to the proof of Th. \ref{tensorProduct}. Let $\pi: \langle A, t\;|\; t^{-1}Bt = B, \; id \rangle \to B(H)$ be an embedding. We define a representation $\pi': \langle A, t\;|\; t^{-1}Bt = B, \; id \rangle\to B(H)$ by
$$\pi'(a)=\pi(a), \;a\in A,$$
$$\pi'(t)= -\pi(t)^*.$$ Let $\tilde \pi = \pi\oplus\pi'$. Since $A$ is MF, by Theorem \ref{characterization} $\pi|_A$ lifts to a discrete asymptotic homomorphism $\phi_n: A \to \mathcal D$, $n\in \mathbb N$. Then $\tilde\pi|_A$ lifts to the discrete asymptotic homomorphism $\tilde \phi_n = \phi_n\oplus \phi_n$, $n\in \mathbb N$.  Let $\{b_1, b_2, \ldots\}$ be a dense subset of $B$. By Lemma \ref{crucial} $\tilde \pi(t)$ lifts to a unitary $V_n$ such that
$$\|[V_n , \tilde\phi_n(b_i)]\|<\frac{1}{n},$$  $i\le n$. By \cite[Lemma 7 + Rem. 8]{sectionsCones}, there exists a discrete asymptotic homomorphism from $\langle A, t\;|\; t^{-1}Bt = B, \; id \rangle$ to $M_2(\mathcal D)$ that lifts $\tilde \pi$. By Theorem \ref{characterization} $\langle A, t\;|\; t^{-1}Bt = B, \; id \rangle$ is MF.
\end{proof}


\begin{corollary} If $C^*(H)$ is MF and $G$ is a right-angled Artin group, then $C^*(H\times G)$ is MF.
 \end{corollary}
\begin{proof} Induction on the number $k$ of vertices of the defining graph $\Gamma$ of $G$. If $k=1$, $C^*(G) = C^*(\mathbb Z)$  and the statement holds by Theorem \ref{tensorProduct}.
Suppose it is proved for $k\le n-1$. Let $v$ be one of the vertices of $\Gamma$, let $\Gamma_1$ be the graph obtained from $\Gamma$ by removing $v$ and all the edges coming from $v$, and let $G_1$ be the right-angled Artin group defined by $\Gamma_1$. Then
$$C^*(H\times G) = \langle C^*(H\times G_1), t\;|\; t^{-1}Bt = B, \; id \rangle,$$ where $B$ is the C*-subalgebra of $C^*(H\times G_1)$ generated by
$C^*(H)$ and $\{w\in V(\Gamma_1)\;|\; [w, v]=0\}$. Theorem \ref{HNN} finishes the proof.
\end{proof}

\begin{remark} It is known that if $G$ is a right-angled Artin group, then $C^*(G)$ is not only MF, but is even quasidiagonal \cite[Th. 4.8]{Atkinson}.
\end{remark}

 \section{Crossed products}

 Let $A$ be a unital C*-algebra and $G$  a group that acts on $A$ by automorphisms. 
 We will identify $A$ and $G$ with their canonical images in $A\rtimes G$.
 Given two groups $G$ and $H$ with a common subgroup $F$, we will identify $G$ and $H$ with their canonical images in $G\ast_F H$.

 \begin{proposition}\label{crossedProduct} $A \rtimes (G\ast_F H) \cong (A\rtimes G)\ast_{A\rtimes F}(A\rtimes H)$.
 \end{proposition}
 \begin{proof} Let $$i_1: A\rtimes G \to (A\rtimes G)\ast_{A\rtimes F}(A\rtimes H), $$
 $$i_2: A\rtimes H \to (A\rtimes G)\ast_{A\rtimes F}(A\rtimes H)$$
 be the canonical inclusions. We define $\alpha_A: A\to (A\rtimes G)\ast_{A\rtimes F}(A\rtimes H)$ by
 $$\alpha_A(a) = i_1(a) \equiv i_2(a),$$ $a\in A$. We define $\alpha_{G\ast_F H}: G\ast_F H \to (A\rtimes G)\ast_{A\rtimes F}(A\rtimes H)$ by
 $$\alpha_{G\ast_F H}(g_1h_1\ldots g_Nh_N) = i_1(g_1)i_2(h_1)\ldots i_1(g_N)i_2(h_N),$$
 $g_1, \ldots, g_N\in G, h_1, \ldots h_N\in H, N\in \mathbb N$. Since the amalgamated C*-subalgebra in $(A\rtimes G)\ast_{A\rtimes F}(A\rtimes H)$ contains both $A$ and $F$, both $\alpha_A$ and $\alpha_{G\ast_F H}$ are well-defined.  We have
 \begin{multline*}\alpha_{G\ast_F H}(g_1h_1\ldots g_Nh_N) \alpha_A(a)\alpha_{G\ast_F H}(g_1h_1\ldots g_Nh_N)^{-1} \\ = i_1(g_1)i_2(h_1)\ldots i_1(g_N)i_2(h_N)i_2(a)i_2(h_N^{-1})i_1(g_N^{-1})\ldots i_2(h_1^{-1})i_1(g_1^{-1}) \\ = i_1(g_1)i_2(h_1)\ldots i_1(g_N)i_1(h_Na)i_1(g_N^{-1})\ldots i_2(h_1^{-1})i_1(g_1^{-1}) \\ = \ldots = i_1(g_1h_1\ldots g_Nh_n a) = \alpha_A(g_1h_1\ldots g_Nh_N a).
 \end{multline*}
 Thus  $(\alpha_A, \alpha_{G\ast_F H})$ is a covariant pair and therefore defines a $\ast$-homomorphism $$\alpha: A \rtimes (G\ast_F H) \to (A\rtimes G)\ast_{A\rtimes F}(A\rtimes H)$$ such that $$\alpha|_A = \alpha_A, \; \alpha|_{G\ast_F H} = \alpha_{G\ast_F H}.$$

 We define $$\beta: (A\rtimes G)\ast_{A\rtimes F}(A\rtimes H) \to A \rtimes (G\ast_F H)$$ by
 $$\beta(i_1(a)) = a,$$
 $$\beta(i_1(g)) = g,$$
 $$\beta(i_2(a)) = a,$$
 $$\beta(i_2(h)) = h,$$
 $a\in A, g\in G, h\in H$. Since $\beta(i_1(a)) = \beta(i_2(a))$ and $\beta(i_1(f)) = \beta(i_2(f))$, $\beta$ is a well-defined $\ast$-homomorphism. It is straightforward to check that $\alpha\circ \beta = id$ and $\beta\circ \alpha = id$.

 \end{proof}

 \begin{remark} \label{crossedProductManyFactors} In the proposition above the amalgamated free product of two factors can be replaced by amalgamated free product of more  than two factors.
 \end{remark}

 \begin{corollary} $C^*(\mathbb Z^2\rtimes SL_2(\mathbb Z))$ is MF.
 \end{corollary}
 \begin{proof} $SL_2(\mathbb Z) \cong \mathbb Z_4\ast_{\mathbb Z_2} \mathbb Z_6$ (\cite{SerreTrees}). Therefore by Proposition \ref{crossedProduct}
 $$\mathbb Z^2\rtimes SL_2(\mathbb Z) \cong \left(\mathbb Z^2\rtimes \mathbb Z_4\right)\ast_{\mathbb Z^2\rtimes \mathbb Z_2} \left(\mathbb Z^2\rtimes \mathbb Z_6\right).$$ Since $Z^2\rtimes \mathbb Z_4$ and $\mathbb Z^2\rtimes \mathbb Z_6$ are amenable, by Theorem \ref{amenable} $C^*(\mathbb Z^2\rtimes SL_2(\mathbb Z))$ is MF.
 \end{proof}


 \begin{corollary} For any semidirect product $G\rtimes  \mathbb F_n$ of an amenable group  $G$ by a free group $ \mathbb F_n$,  $C^*(G\rtimes \mathbb F_n)$ is MF.
  \end{corollary}
  \begin{proof} By Proposition \ref{crossedProduct} and Remark \ref{crossedProductManyFactors},
  $G\rtimes \mathbb F_n$  is a free product of $G\rtimes \mathbb Z$ amalgamated over $G$.   Since $G\rtimes \mathbb Z$ is amenable, the statement follows from Theorem \ref{amenable} and Remark \ref{remarkManyAmenbaleFactors}.
  \end{proof}

\end{document}